\documentclass[11pt]{article}
\usepackage[utf8]{inputenc}

\title{The equivariant Ehrhart theory of \\
polytopes
with order-two symmetries}
\author{Oliver Clarke\thanks{Oliver Clarke is an overseas researcher under Postdoctoral Fellowship of Japan Society for the Promotion of Science (JSPS).}, Akihiro Higashitani, and Max K\"olbl}
\date{\empty}

\usepackage{fullpage}
\usepackage[square,sort,comma,numbers]{natbib}

\usepackage{amsthm}
\usepackage{amssymb}
\usepackage{amsmath}
\usepackage{pdfpages}
\usepackage{booktabs}
\usepackage{multirow}
\usepackage{graphics}
\usepackage{stmaryrd}
\usepackage{mathtools}
\usepackage{tikz}

\usepackage[bookmarksnumbered=true]{hyperref}
\hypersetup{
     colorlinks = true,
     linkcolor = blue,
     anchorcolor = blue,
     citecolor = teal,
     filecolor = blue,
     urlcolor = blue
     }

\theoremstyle{plain}

\newtheorem{theorem}{Theorem}[section]

\newtheorem{conjecture}[theorem]{Conjecture}
\newtheorem{proposition}[theorem]{Proposition}

\newtheorem{lemma}[theorem]{Lemma}

\theoremstyle{definition}

\newtheorem{example}[theorem]{Example}

\theoremstyle{remark}

\newtheorem{remark}[theorem]{Remark}

\DeclarePairedDelimiter\abs{\lvert}{\rvert}
\DeclarePairedDelimiter\scalar{\langle}{\rangle}
\DeclarePairedDelimiter{\set}{\{}{\}}
\DeclarePairedDelimiter\ceil{\lceil}{\rceil}
\DeclarePairedDelimiter\floor{\lfloor}{\rfloor}

\newcommand{\RR}{\mathbb R}
\newcommand{\ZZ}{\mathbb Z}

\newcommand{\QQ}{\mathbb Q}

\newcommand{\conv}{{\rm Conv}}

\usepackage{xcolor}

\DeclareMathOperator{\diag}{Diag}
\DeclareMathOperator{\ehr}{Ehr}
\DeclareMathOperator{\GL}{GL}
\DeclareMathOperator{\tr}{tr}
\DeclareMathOperator{\aut}{Aut}

\makeatletter
\@ifdefinable\@latex@chi{\let\@latex@chi\chi}
\renewcommand*\chi{{\@latex@chi\smash[t]{\mathstrut}}} 
\makeatletter

\begin{document}

\maketitle

\begin{abstract}
    We study the equivariant Ehrhart theory of families of polytopes that are invariant under a non-trivial action of the group with order two. We study families of polytopes whose equivariant $H^\ast$-polynomial both succeed and fail to be effective, in particular, the symmetric edge polytopes of cycle graphs and the rational cross-polytope. The latter provides a counterexample to the effectiveness conjecture if the requirement that the vertices of the polytope have integral coordinates is loosened to allow rational coordinates. Moreover, we exhibit such a counterexample whose Ehrhart function has period one and coincides with the Ehrhart function of a lattice polytope.
\end{abstract}

\section{Introduction}

Ehrhart theory is the enumerative study of the lattice points of polytopes and their dilations \cite[Section~3]{beck2007computing}. Let $P \subseteq \RR^d$ be a polytope. The \textit{Ehrhart function} $L_P(m) = |mP \cap \ZZ^d|$ with $m \in \ZZ_{\geq 0}$ counts the number of lattice points of $mP$. If $P$ is a lattice polytope then $L_P(m)$ is a polynomial called the \textit{Ehrhart polynomial of $P$}.
More generally, if $P$ is a rational polytope, i.e.~its vertices have rational coordinates, then $L_P(m)$ becomes a quasipolynomial. We say that $Q(t)=c_d(t)t^d+\cdots+c_1(t)t+c_0(t)$ is a \textit{quasipolynomial} if $c_0,\ldots,c_d$ are periodic functions in $t$ and define the \textit{period of $Q(t)$} to be the least common multiple of the periods of $c_0,\ldots,c_d$. Note that a usual polynomial can be regarded as a quasipolynomial with period one.
The data of $L_P$ is expressed as a power series $\ehr(P, t) = \sum_{m \geq 0} L_P(m) t^m$ called the \textit{Ehrhart series}. If $P$ is a rational polytope, then
\[
\ehr(P, t) = \frac{h^\ast_P(t)}{(1-t^N)^{d+1}}
\]
for some polynomial $h^\ast_P(t) \in \ZZ[t]$ called the $h^\ast$-polynomial of $P$. The value of $N$ is the \textit{denominator} of $P$, which is defined as the smallest positive integer $\ell$ such that $\ell P$ is a lattice polytope. It is well known that the denominator of $P$ is divisible by the period of $L_P(t)$. While all lattice polytopes have period one, the converse is not true. A rational polytope $P$ whose Ehrhart quasipolynomial has period one is said to be a \textit{pseudo-integral polytope} or \textit{PIP}. We will see examples of PIPs in Section~\ref{sec: rational cross polytopes}.

In \cite{stapledon2011equivariant}, Stapledon introduces a generalisation of Ehrhart theory to study polytopes that exhibit symmetries. The theory has connections to toric geometry, representation theory, and mirror symmetry. Suppose that the polytope $P$ is invariant (up to translation) under the action of a finite group $G$ acting linearly on the lattice by a representation $\rho : G \rightarrow \GL(\ZZ^d)$. The \textit{equivariant Ehrhart series} $\ehr_\rho(P, t) \in R(G)[[t]]$ is a power series in $t$ with coefficients in the representation ring $R(G)$.
The series $\ehr_\rho(P, t)$ can be thought of as a union of the Ehrhart series of fixed sub-polytopes of $P$. Explicitly, for each $g \in G$, the power series $\ehr_\rho(P, t)(g) \in \ZZ[[t]]$ is the Ehrhart series for the sub-polytope of $P$ fixed by $g$. In particular, $\ehr_\rho(P, t)(1_G) = \ehr(P, t)$ recovers the original Ehrhart series. See Remark~\ref{rmk: prelim setup evaluate at identity}.

The analogue for the $h^\ast$-polynomial in equivariant Ehrhart theory is the equivariant $H^\ast$-series denoted $H^\ast[t] = \sum_{i \geq 0} H^\ast_i t^i \in R(G)[[t]]$.
In older literature, it is also denoted as $\varphi[t]$.
In general, this series is not a polynomial. However, one of the central questions in equivariant Ehrhart theory is to determine when $H^\ast[t]$ is a polynomial and how to interpret its coefficients $H^\ast_i \in R(G)$. Recall that for a lattice polytope $P$, the coefficients of its $h^\ast$-polynomial are non-negative \cite{stanley1980decompositions}. The equivariant analogue for non-negativity is effectiveness. We say that $H^\ast[t]$ is \textit{effective} if each of its coefficients $H^\ast_i$ is a non-negative integer sum of irreducible representations.

\begin{conjecture}[{Effectiveness conjecture \cite[Conjecture~12.1]{stapledon2011equivariant}}]\label{effectiveness_conjecture}
Fix the main setup from Section~\ref{sec: prelim main setup} and let $P$ be a lattice polytope. The equivariant $H^\ast$-series $H^\ast[t]$ is a polynomial if and only if $H^\ast[t]$ is effective.
\end{conjecture}

It is known that if $H^\ast[t]$ is effective, then it is a polynomial. However the converse is currently open. The equivariant Ehrhart theory for certain families of polytopes is well studied. In each of the following examples the effectiveness conjecture has been verified: simplices \cite[Proposition~6.1]{stapledon2011equivariant}, for which the coefficients of $H^\ast[t]$ are permutation representations; the hypercube \cite[Section~9]{stapledon2011equivariant} under its full symmetry group; the permutahedron \cite{ardila2019permutahedron} under the symmetric group; graphic zonotopes of the path graph \cite[Section~3.1]{elia2022techniques} with the $\ZZ/2\ZZ$ action; hypersimplices \cite[Theorem~3.57]{elia2022techniques} with the symmetric group action.

In this paper we study two new families of polytopes; the symmetric edge polytopes of the cycle graph under the induced action of the automophism group of the graph, and rational cross-polytopes under the action of coordinate reflections. We describe the fixed polytopes in each case, which are related to rational cross-polytopes. We compute the equivariant Ehrhart series in each case to verify the effectiveness conjecture. In particular, in Example~\ref{ex: cross polytope rational non effective} we see that PIP need not satisfy the effectiveness conjecture if the assumption that $P$ is a lattice polytope is dropped.

\noindent \textbf{Outline.} In Section~\ref{sec: prelim} we introduce the necessary preliminaries with the aim of fixing the main setup for equivariant Ehrhart theory. In Section~\ref{sec: prelim representations of groups}, we recall some basics of representation theory of finite groups, in particular the representation ring $R(G)$ which serves as the coefficient ring for the equivariant Ehrhart series. In Section~\ref{sec: prelim actions on lattices}, we fix our notation for actions of groups on lattices and describe their affine lattices. In Section~\ref{sec: prelim main setup}, we recall the main setup of equivariant Ehrhart theory. In Remark~\ref{rmk: prelim alternative setup}, we also give an alternative but equivalent setup.

In the next sections, we analyse two families of symmetric polytopes. In Section~\ref{sec: symmetric edge polytopes}, we consider the symmetric edge polytopes of cycle graphs under symmetry induced by the dihedral group acting on the graph. We prove Theorems~\ref{thm: prime cycles} and \ref{thm: symm edge poly order 2 group} which show that Conjecture \ref{effectiveness_conjecture} holds in the following respective cases. Firstly, for the cycle graph with a prime number of vertices under the action of its full automorphism group and secondly for any cycle graph with at least three vertices under the action of a reflection. In Section~\ref{sec: rational cross polytopes}, we consider a family of rational cross-polytopes under the group of coordinate reflections. We prove Theorem~\ref{thm: cross polytope coord reflection} which computes the equivariant $H^\ast$-series for all polytopes we consider. We note that this family contains rational polytopes with non-effective polynomial $H^\ast$-series; see Example~\ref{ex: cross polytope rational non effective}.

\noindent \textbf{Acknowledgements.} We would like to express our gratitude to the anonymous reviewers for paying a great deal of attention to this paper and supplying many helpful comments and suggestions.

\section{Preliminaries}\label{sec: prelim}

Equivariant Ehrhart theory concerns the study of polytopes and their lattice points under a given group action. In this section we introduce the necessary preliminaries and fix the main setup following \cite{stapledon2011equivariant}. We begin with some background on the representation theory of finite groups \cite{isaacs1994character, curtis1966representation}.

\subsection{Representations of groups}\label{sec: prelim representations of groups}
Let $G$ be a finite group and $K$ a field. A \textit{finite dimensional $K$-representation of $G$} is a homomorphism $\rho: G \rightarrow \GL(V)$ from $G$ to the group of invertible linear maps of an $n$-dimensional $K$-vector space $V$. Fixing a basis for $V$ identifies $\rho(g)$ with an $n \times n$ matrix, for each $g \in G$. Equivalently, a representation is a module $V$ for the group ring $KG$ where $g \in G \subseteq KG$ {acts via} the linear map $\rho(g)$. The \textit{character of $\rho$} is the function $\chi : G \rightarrow K$ defined by the trace $\chi(g) = \tr(\rho(g))$. We say that a representation is \textit{irreducible} if it contains no proper $G$-invariant subspaces, \textit{indecomposable} if it cannot be written as a non-trivial direct sum of representations, and \textit{semisimple} if it is a direct sum of irreducible representations.

The \textit{representation ring} $R(G)$ is the set of formal differences of isomorphism classes of representations of $G$. The addition and multiplication structure of $R(G)$ are given by direct sums and tensor products respectively. Given a $KG$-module $V$, we write $[V]$ for its isomorphism class in $R(G)$. So given $[V]$ and $[W]$ in $R(G)$ we have $[V]+[W] = [V \oplus W]$ and $[V]\cdot[W] = [V \otimes_K W]$. In this paper, we work with representations defined over $\RR$. In this case Maschke's Theorem holds, so all representations are semisimple. In particular, all indecomposable representations are irreducible and any representation is a direct sum of irreducible representations. Therefore, $R(G)$ is a free abelian group generated by the irreducible representations of $G$. Since the isomorphism class of a representation is determined uniquely by its character, we identify elements of $R(G)$ with $\ZZ$-linear combinations of characters.

\noindent \textbf{Permutation representations.} Suppose $G$ acts on a finite set $S$. Then the action induces a so-called \textit{permutation representation} constructed as follows. Let $V$ be the vector space over some field $K$ with basis $\{e_s : s \in S\}$. We define the permutation representation $\rho : G \rightarrow \GL(V)$ by its action on the basis $\rho(g)(e_s) = e_{g(s)}$. Each matrix $\rho(g)$ is a permutation matrix, hence the character of the representation is given by $\chi(g) = |\{s \in S : g(s) = s\}|$. We say that a $KG$-module $V$ is a permutation representation if it is isomorphic to a permutation representation.

\subsection{Group actions on lattices} \label{sec: prelim actions on lattices}

Let $M \cong \ZZ^{n+1}$ be a lattice with a distinguished basis and  $G$ a finite group. We say that \textit{$G$ acts on $M$} if there is a homomorphism $\rho : G \rightarrow \GL_{n+1}(\ZZ)$ from $G$ to the group of invertible $(n+1) \times (n+1)$ matrices with entries in $\ZZ$. Note, this action extends naturally to the vector space $M_\RR = M \otimes_\ZZ \RR$. Assume that $G$ fixes a lattice point $e \in M \backslash \{0\}$. We proceed to describe how $M$ decomposes into a disjoint union of $G$-invariant affine lattices.

By assumption $M$ has a basis, so we denote by $\langle \cdot, \cdot \rangle : M \times M \rightarrow \ZZ$ the standard inner-product. We construct a new inner-product by averaging over the group:
\[
\langle u, v \rangle_G := \frac{1}{|G|} \sum_{g \in G} \langle \rho(g)u, \rho(g)v\rangle \in \QQ.
\]
Using the above inner-product, we observe two important properties about the orthogonal space $e^\perp \subseteq M_\RR$. Firstly, we have that $e^\perp$ is $G$-invariant, which follows from the fact that $\langle\rho(g)u, \rho(g)v\rangle_G = \langle u,v \rangle_G$ for all $u,v \in M_\RR$ and $g \in G$. Secondly, we may choose a basis for $e^\perp$ that lies in $M$, since $\langle u, v\rangle_G \in \QQ$ for all $u,v \in M$. It follows that the lattice $N$ generated by $e^\perp \cap M$ and $e$ has rank $n+1$. Therefore, $N$ is a finite index subgroup of $M$ and we write $[M:N]$ for the index.
We define the \textit{affine space $(M_i)_\RR$} and the \textit{affine lattice $M_i$} at height $i \in \ZZ$ as follows:
\[
(M_i)_\RR = \frac{i}{[M : N]} e + e^\perp
\quad \text{and} \quad
M_i = (M_i)_\RR \cap M.
\]
Since $e^\perp$ and $M$ are $G$-invariant, we have that $M_i$ is $G$-invariant for each $i \in \ZZ$. Note that $M = \bigcup_{i \in \ZZ} M_i$ is a disjoint union and for each $v \in M_i$ we have $v + M_j = M_{i+j}$.

\begin{example}\label{ex: prelim action on lattice}
Let $G = \{1, \sigma \} \le S_4$ be a subgroup of the symmetric group on four letters with $\sigma = (1,2)(3,4)$. The permutation representation $\rho$ maps $\sigma$ to the permutation matrix
\[
\rho(\sigma) =
\begin{bmatrix}
0 & 1 & 0 & 0 \\
1 & 0 & 0 & 0 \\
0 & 0 & 0 & 1 \\
0 & 0 & 1 & 0
\end{bmatrix}
\in \GL_4(\RR).
\]
In particular, this matrix lies in $\GL_4(\ZZ)$, hence $G$ preserves the lattice $M = \ZZ[e_1, e_2, e_3, e_4]$. Notice that $e = e_1+e_2+e_3+e_4$ is fixed by the action of $G$. We compute a basis $F$ that decomposes $\rho(\sigma)$ as a block diagonal matrix:
\[
F =
\left\{
\begin{bmatrix}
1 \\1 \\1 \\1 \\
\end{bmatrix}, \
\begin{bmatrix}
1 \\-1 \\0 \\0 \\
\end{bmatrix}, \
\begin{bmatrix}
1 \\0 \\-1 \\0 \\
\end{bmatrix}, \
\begin{bmatrix}
1 \\0 \\0 \\-1 \\
\end{bmatrix}
\right\} \quad \text{and} \quad
\rho(\sigma)_F =
\begin{bmatrix}
1 & 0 & 0 & 0 \\
0 & -1 & -1 & -1 \\
0 & 0 & 0 & 1 \\
0 & 0 & 1 & 0
\end{bmatrix}.
\]
The orthogonal lattice $M_0$ is the $3$-dimensional lattice generated by $F \backslash \{e\}$. Observe that the sublattice $N = \ZZ[F]$ has index $4$ inside $M$. Therefore, the affine lattice $M_1 = (\frac 14 e + (M_0)_\RR) \cap M$ is equal to the lattice affinely generated by $\{e_1, e_2, e_3, e_4 \}$.
\end{example}

\subsection{Main setup} \label{sec: prelim main setup}

Let $M \cong \ZZ^{n+1}$ be a lattice with a distinguished basis and $G$ a finite group that acts on $M$ by $\rho : G \rightarrow \GL_{n+1}(\ZZ)$. Assume that there is a lattice point $e \in M \backslash \{0\}$ fixed by $G$. Let $P \subseteq (M_1)_\RR$ be a rational $G$-invariant polytope. For each non-negative integer $m \in \ZZ_{\geq 0}$, we obtain a permutation representation of the lattice points $mP \cap M \subseteq M_m$ and denote by $\chi_{mP}$ its character. The \textit{equivariant Ehrhart series} is an element of the ring of formal power series $R(G)[[t]]$ given by:
\[
\ehr_\rho(P, t) =
\sum_{m \geq 0} \chi_{mP} t^m =
\frac{H^\ast[t]}{\det[I - t \cdot \rho]} =
\frac{H^\ast[t]}{(1-t)\det[I - t \cdot \rho|_{M_0}]}
\]
where $H^\ast[t] \in R(G)[[t]]$ is the \textit{equivariant $H^\ast$-series}. The denominator $\det[I - t \cdot \rho]$ denotes the formal alternating sum
$\sum_{i = 0}^{n+1} [\Lambda^{i}M_{\RR}]  (-t)^i \in R(G)[t],$
where $\Lambda^i M_\RR$ is the $i$-th alternating power of the representation $M_\RR$. If the character of the above alternating sum is evaluated at an element $g \in G$, then the resulting polynomial is equal to $\det[I - t \cdot \rho(g)]$ where $I$ is the identity matrix, see \cite[Lemma~3.1]{stapledon2011equivariant}.

By assumption, $M_\RR = \langle e \rangle_\RR \oplus (M_0)_\RR$ is a $G$-invariant decomposition of $M_\RR$. So, for each $g \in G$, we may write $\rho(g) = [1] \oplus \rho(g)|_{M_0}$ as a block diagonal matrix, hence $\det[I - t \cdot \rho(g)] = (1-t)\det[I - t \cdot \rho(g)|_{M_0}]$.

\begin{remark}\label{rmk: prelim setup evaluate at identity}
The equivariant Ehrhart series and $H^\ast$-series are a generalisation of the usual Ehrhart series and $h^\ast$-polynomial. If the equivariant Ehrhart series is evaluated at the identity element, then each character $\chi_{mP}(1_G)$ is equal to the number of lattice points of $mP$. Since $\det[I - t\cdot \rho(1_G)] = (1-t)^{n+1}$, it follows that the equivariant Ehrhart series evaluated at $1_G$ is equal to the classical Ehrhart series $\ehr(P,t)$.

The equivariant Ehrhart series contains all the data about the Ehrhart series for fixed sub-polytopes of $P$. Let $M_\RR^g = \{x \in M_\RR : g(x) = x \}$ be the subspace of $M_\RR$ fixed by $g \in G$. 
For each $m \geq 0$ and $g \in G$, the value $\chi_{mP}(g)$ is the number of lattice points of $mP$ fixed by $g$. Equivalently, $\chi_{mP}(g)$ is the number of lattice points in the $m$-th dilate of the \textit{fixed polytope} $P^g = P \cap M_\RR^g$. Therefore, the evaluation of the equivariant Ehrhart series at $g \in G$ is the Ehrhart series $\ehr(P^g, t)$.
\end{remark}

\begin{remark}\label{rmk: prelim alternative setup}
The setup may be equivalently defined by fixing: a group action $\rho|_{M_0}$ of $G$ on a lattice $M_0 \cong \ZZ^n$; a rational polytope $P \subseteq (M_1)_\RR$, where $M_1 \cong \ZZ^n$ is a lattice of the same rank; and a lattice-preserving isomorphism between $(M_1)_\RR$ and $(M_0)_\RR$, which induces an action of $G$ on $P$. We require that, for each $g \in G$, the polytope $g(P) = (-v_g) + P$ differs from $P$ only by a translation $v_g \in M_0$. So, for all $g, h \in G$ we have that 
\[
(gh)(P) + v_{gh} = P = g(P) + v_g = g(h(P) + v_h) + v_g = (gh)(P) + g(v_h) + v_g,
\]
hence $v_{gh} = g(v_h) + v_g$.

We recover the original setup by taking $e \in |G|\cdot P \subseteq (M_{|G|})_\RR$ to be any $G$-invariant lattice point of the $|G|$th dilate of $P$. Explicitly, for all $g \in G$, we require $g(e) + |G| \cdot v_g = e$. For example, such a point can always be constructed from any lattice point $p \in P$ by summing over the group: $e = \sum_{g \in G} \left( g(p) + v_g \right)$.
We define $M$ to be the lattice generated by $M_0$ and $M_1$ where $M_0$ is a lattice that contains the origin and $M_1$ is the affine lattice at height $1$ such that the orthogonal projection of $(M_1)_\RR$ onto $(M_0)_\RR$ sends $\frac{1}{|G|}e \in (M_1)_\RR$  to $0 \in M_0$ and differs from the lattice-preserving isomorphism by a translation. Concretely, we may take $M = \ZZ \times M_0 \cong \ZZ^{n+1}$ and define the action of $G$ on $M$ by the matrix
$\rho(g) = 
\begin{bmatrix}
1 & 0 \\
v_g & \rho|_{M_0}(g)
\end{bmatrix}.
$
Note that $\rho$ is indeed a group homomorphism. That is, for all $g$ and $h$ in $G$ we have
\[
\rho(g) \rho(h)
= 
\begin{bmatrix}
1 & 0 \\
v_{g} & \rho|_{M_0}(g)
\end{bmatrix}
\begin{bmatrix}
1 & 0 \\
v_{h} & \rho|_{M_0}(h)
\end{bmatrix}
=
\begin{bmatrix}
1 & 0 \\
g(v_h) + v_g & \rho|_{M_0}(gh)
\end{bmatrix}
=
\rho(gh)
\]
since $g(v_h) + v_g = v_{gh}$.

Let $\lambda \in \ZZ_{>0}$ be the smallest positive integer such that $\frac{\lambda}{|G|} e$  is a lattice point. The value of $\lambda$ coincides with the index of the sublattice $N$ in $M$ from the original setup.
\end{remark}

\begin{example}[Continuation of Example~\ref{ex: prelim action on lattice}]\label{ex: simplex with a double reflection}
Recall $G = \{1, \sigma\} \le S_4$, with $\sigma = (1,2)(3,4)$, acting by a permutation representation on $M = \ZZ^4$. Let $P = \conv\{e_1, e_2, e_3, e_4\} \subseteq (M_1)_\RR$ be a $G$-invariant $3$-dimensional simplex. The permutation character $\chi_{mP}$ counts the number of lattice points of $mP \subseteq M_m$ fixed by each $g \in G$. Explicitly, we have
\[
\chi_{mP}(1) = \binom{m+3}{3} \quad \text{and} \quad \chi_{mP}(\sigma) =
\begin{cases}
\frac m2 + 1 & \text{if } 2 \mid m, \\
0 & \text{otherwise.}
\end{cases}
\]
Computing the equivariant Ehrhart series, we have
\[
\sum_{m \geq 0} \chi_{mP}(1)t^m = \frac{1}{(1-t)^4}
\quad \text{and} \quad
\sum_{m \geq 0} \chi_{mP}(\sigma)t^m = \frac{1}{(1-t^2)^2}.
\]
For each $g \in G$, we observe that the equivariant Ehrhart series is given by $\frac{1}{\det[I - t\cdot\rho(g)]}$. Therefore, the equivariant $H^\ast$-series is a polynomial given by $H^\ast[t] = 1$.
\end{example}

\begin{example}
Following the alternative setup in Remark~\ref{rmk: prelim alternative setup}, let $G = \{1, \sigma \}$ be the group with two elements that acts on a rank $3$ lattice $M_0 = \ZZ[e_1, e_2, e_3]$ by the map
\[
\sigma \mapsto
\begin{bmatrix}
-1 & -1 & -1 \\
0 & 0 & 1 \\
0 & 1 & 0
\end{bmatrix}.
\]
Let $P = \conv\{0,e_1,e_2,e_3\}$ and notice that $\sigma(P) = (-e_1) + P$, hence the above map defines a valid setup.
This setup is equivalent to the setup in Example~\ref{ex: simplex with a double reflection}, which can be seen as follows. By averaging the vertex $0 \in P$ over $G$, we obtain the $G$-invariant point $e = \frac 12 e_1$, verified by the fact that $e = \sigma(e) + e_1$. We define the lattice $M = \ZZ[e_0, e_1, e_2, e_3]$ and identify the affine sublattice of $M$ containing $P$ with the affine span of $\{e_0 + e_1, e_0 + e_2, e_0 + e_3\}$. In particular, the polytope $P$ is identified in $M_\RR$ as $\conv\{e_0, e_0+e_1, e_0+e_2, e_0+e_3 \}$. The action of $G$ on $P$ extends to an action of $G$ on $M$ given by
\[
\sigma \mapsto
\begin{bmatrix}
1 & 0 & 0 & 0 \\
1 & -1 & -1 & -1 \\
0 & 0 & 0 & 1 \\
0 & 0 & 1 & 0
\end{bmatrix}.
\]
The point $e$ in $M_\RR$ is identified with $e_0 + \frac 12 e_1$ which spans a $1$-dimensional $G$-invariant subspace. Observe that the vertices of $P \subseteq M_\RR$ are a basis for the lattice $M$. Rewriting the action of $G$ in terms of this basis identifies it with Example~\ref{ex: simplex with a double reflection}.
\end{example}

\smallskip

\noindent \textbf{Effectiveness of the equivariant $H^\ast$-series.} We say that the equivariant $H^\ast$-series $H^\ast[t] = \sum_{i \geq 0} H^\ast_i t^i \in R(G)[[t]]$ is \textit{effective} if each $H^\ast_i \in R(G)$ is the isomorphism class of a representation of $G$. In other words, $H^\ast_i$ is a non-negative sum of irreducible representations of $G$. One of the main problems in equivariant Ehrhart theory is to understand when $H^\ast[t]$ is effective.

\begin{conjecture}[{\cite[Conjecture~12.1]{stapledon2011equivariant}}]
Let $G$ be a finite group that acts on a lattice and $P$ a $G$-invariant lattice polytope. Let $Y$ be the toric variety with ample line bundle $L$ associated to $P$. Then the following are equivalent:
\begin{itemize}
    \item[$(1)$] $L$ admits a $G$-invariant section that defines a non-degenerate hypersurface of $Y$,
    \item[$(2)$] $H^\ast[t]$ is effective,
    \item[$(3)$] $H^\ast[t]$ is a polynomial.
\end{itemize}
\end{conjecture}

It is well known that $(1) \Rightarrow (2) \Rightarrow (3)$, see \cite{stapledon2011equivariant}, and a counterexample has been constructed by Santos and Stapledon \cite[Theorem~1.2]{elia2022techniques} showing that $(2) \nRightarrow (1)$ and $(3) \nRightarrow (1)$. It is currently open whether $(3) \Rightarrow (2)$.

\section{Symmetric edge polytopes of cycle graphs}\label{sec: symmetric edge polytopes}
In this section we consider symmetric edge polytopes coming from cycle graphs and show that Conjecture~\ref{effectiveness_conjecture} holds for the action of the dihedral group (Theorem~\ref{thm: prime cycles}) if the cycle graphs have prime order, and for the action of the two element subgroups of the dihedral group (Theorem~\ref{thm: symm edge poly order 2 group}) for cycle graphs of any order.
We begin by fixing the setup for this section.
Then, we consider the fixed polytopes of certain  symmetric edge polytopes.
Lastly we conclude the section with the statements and proofs of the main theorems.

Let $\Gamma=(V,E)$ be an undirected graph and $\ZZ^{\abs{V}}$ a lattice whose basis elements $e_v$ are associated to the vertices $v\in V$.
Then the \emph{symmetric edge polytope} $P_\Gamma\subset \RR^{\abs{V}}$ associated to $\Gamma$ is defined as follows:
\[
    P_\Gamma = \conv\left\{\pm(e_v - e_w) \colon \{v,w\}\in E \right\}.
\]
Throughout this section, we shall consider the automorphism group of $\Gamma$, denoted $\aut(\Gamma)$.
One sees that $\aut(\Gamma)$ naturally induces a permutation representation $\rho_\Gamma$ on $\RR^{\abs{V}}$, which leaves $P_\Gamma$ invariant.
We focus on the case when $\Gamma$ is the cycle graph $C_d$ for some integer $d\geq 3$.
In this case, $\aut(C_d)\cong D_{2d}=\scalar{r,s \mid s^2=r^d=(sr)^2=1}$ is the dihedral group of order $2d$.

We identify $D_{2d}$ with the automorphism group of $C_d$.
We fix the generator $s \in D_{2d}$, in the presentation of the group, to be a reflection that fixes the fewest number of vertices of $C_d$. Let $\rho_d:=\rho_{C_d} : D_{2d} \rightarrow \GL(\RR^{d})$ denote the associated permutation representation.
From now on, we label the vertices of $C_d$ with $\set{v_0,\ldots, v_{\ceil{(d-2)/2}},w_0,\ldots, w_{\ceil{(d-2)/2}}}$, where $w_0=v_0$ if $d$ is odd, so that: $(v_0,v_1,\ldots,v_{\ceil{(d-2)/2}})$ and $(w_0,w_1,\ldots,w_{\ceil{(d-2)/2}})$ are distinct paths in $C_d$; for each $0 \le i \le \ceil{(d-2)/2}$ the $s$-orbits are $\{v_i, w_i \}$; if $d$ is odd, then $v_0 = w_0$ is the unique fixed vertex of $s$; if $d$ is even, then $v_0$ and $w_0$ are neighbours; and $r$ is the rotation that maps $w_0$ to $w_1$ (see Figure~\ref{labeling}).
\begin{figure}[h]
    \centering
    \begin{tikzpicture}
        \node[shape=circle,draw=black] (v0) at (0,0) {$\quad v_0\quad$};
        \node[shape=circle,draw=black] (w0) at (0,3) {$\quad w_0\quad$};
        \node[shape=circle,draw=black] (w1) at (2.5,4) {$\quad w_1\quad$};
        \node[shape=circle,draw=black] (v1) at (2.5,-1) {$\quad v_1\quad$};
        \node[shape=circle,draw=black] (wd) at (5,3) {$w_{(d-2)/2}$};
        \node[shape=circle,draw=black] (vd) at (5,0) {$v_{(d-2)/2}$};

        \path [-](w0) edge node[left] {} (w1);
        \path [-](v0) edge node[left] {} (v1);
        \path [dotted](w1) edge node[left] {} (wd);
        \path [dotted](v1) edge node[left] {} (vd);
        \path [-](vd) edge node[left] {} (wd);
        \path [-](w0) edge node[left] {} (v0);
        
        \node[] (s1) at (5.75,1.75) {$s$};
        \path[dashed](-1,1.5) edge node[left] {} (6,1.5);
        
        \node[] (s2) at (14.25,1.75) {$s$};
        \path[dashed](9.35,1.5) edge node[left] {} (14.5,1.5);
        
        \node[] (r1) at (1.5,3) {$r$};
        \draw[dashed, ->] (0.8,2.25)  arc[start angle=140, end angle=90, radius=2cm];
        
        \node[] (r2) at (10,2.6) {$r$};
        \draw[dashed, ->] (9.8,1.9)  arc[start angle=160, end angle=110, radius=1.5cm];

        \node[shape=circle,draw=black] (v0o) at (8.5,1.5) {$v_0=w_0$};
        \node[shape=circle,draw=black] (w1o) at (11,4) {$\quad w_1\quad$};
        \node[shape=circle,draw=black] (v1o) at (11,-1) {$\quad v_1\quad$};
        \node[shape=circle,draw=black] (wdo) at (13.5,3) {$w_{(d-1)/2}$};
        \node[shape=circle,draw=black] (vdo) at (13.5,0) {$v_{(d-1)/2}$};

        \path [-](v0o) edge node[left] {} (w1o);
        \path [-](v0o) edge node[left] {} (v1o);
        \path [dotted](w1o) edge node[left] {} (wdo);
        \path [dotted](v1o) edge node[left] {} (vdo);
        \path [-](vdo) edge node[left] {} (wdo);
    \end{tikzpicture}
    \caption{The vertex labelings for even (left) and odd (right) cycle graphs and the action of the generators of the dihedral group.}\label{labeling}
\end{figure}
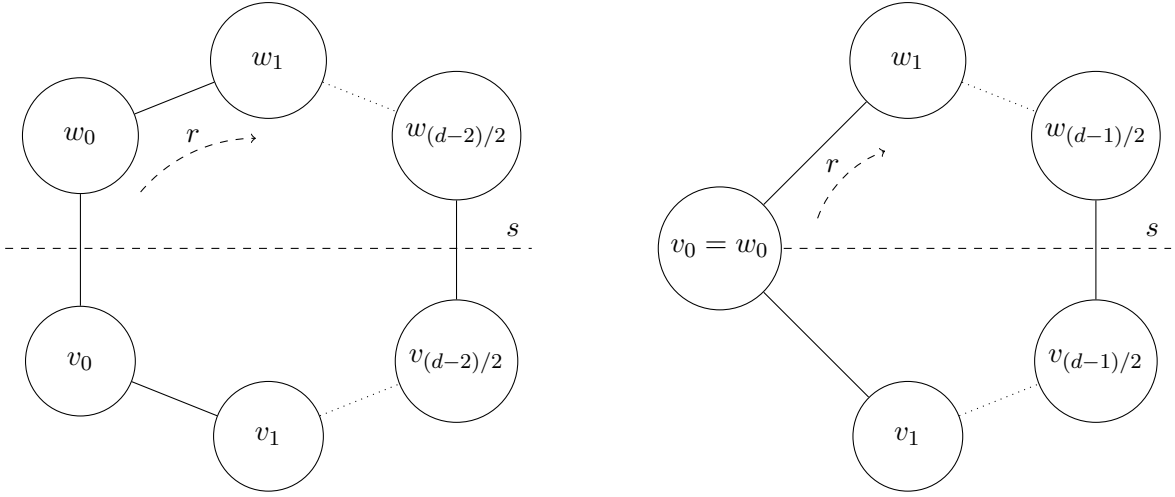

Studying the equivariant Ehrhart theory of $P_d:=P_{C_d}$ under the action of $D_{2d}$ involves understanding the Ehrhart series of the individual sub-polytopes $P_d^g$ fixed by the individual elements $g \in D_{2d}$.
Let us begin with the trivial element $1 \in D_{2d}$.

\begin{proposition}[{\cite[Theorem~3.3]{ohsugi2012smooth}}]\label{prop: ohsugi}
The Ehrhart series of $P_d$ is given by
\[
\ehr(P_d, t) = \frac{h_0^{(d)} + h_1^{(d)} t + \cdots + h_{d-1}^{(d)} t^{d-1}}{(1-t)^d}
\]
where: $h_0^{(d)}=1$; for $1\leq j \leq \floor{\frac{d}{2}}$, we have
\[
h_j^{(d)} = (-1)^j \sum_{i=0}^j (-2)^i \binom{d}{i}\binom{d-1-i}{j-i}
= \begin{cases}
2^{d-1} & \text{if $d$ is odd and $j={\frac{d-1}{2}}$}, \\
h_{j-1}^{(d-1)} + h_j^{(d-1)} & \text{otherwise};
\end{cases}
\]
and for each $\frac{d}{2}<j< d$, the coefficients are $h_j^{(d)} = h_{d-1-j}^{(d)}$.
\end{proposition}

For odd cycle graphs $C_{2\ell+1}$, all reflections in $D_{4\ell+2}$ are conjugate and so the corresponding fixed polytopes are unimodularly equivalent. Hence, it suffices to compute the fixed polytope for a single reflection, say $s\in D_{4\ell+2}$.

\begin{proposition}\label{prop: odd_cycles}
Let $\ell\geq 1$ be an integer.
The fixed sub-polytopes $P_{2\ell+1}^s$ and $P_{2\ell+2}^s$ are unimodularly equivalent to the cross-polytope of dimension $\ell$ dilated by the factor $\frac{1}{2}$ and their Ehrhart series are given by
\[
\ehr(P_{2\ell+1}^s, t) = \ehr(P_{2\ell+2}^s, t) = \frac{(1+t^2)^\ell}{(1-t)(1-t^2)^\ell}.
\]
\end{proposition}

\begin{proof}
We start by giving a full description of the vertices of $P_{2\ell+1}^s$ and $P_{2\ell+2}^s$.
Each $s$-orbit is given by $\{v_i, w_i \}$ for each $0 \le i \le \ell$.
Note, the $s$-orbit that is an edge of $C_{2\ell+1}$ is $\{v_\ell, w_\ell \} \in E$, while those of $C_{2\ell+2}$ are $\{v_0,w_0\}$ and $\{v_\ell,w_\ell\}$.
The $s$-orbits of the vertices of $P_{2\ell+1}$ and $P_{2\ell+2}$ are hence given by $\set{\pm (e_{w_i}- e_{w_{i+1}}), \pm (e_{v_i}- e_{v_{i+1}})}$ as well as $\set{e_{w_\ell} - e_{v_\ell}, e_{v_\ell} - e_{w_\ell}}$.
In the case of $P_{2\ell+2}$, we have the additional vertex $\set{e_{w_0} - e_{v_0}, e_{v_0} - e_{w_0}}$.
By Lemma~5.4 in \cite{stapledon2011equivariant}, $P_{2\ell+1}^s$ (resp. $P_{2\ell+2}$) is given by the convex hull of points of the form $\frac{\sum_{p\in I}p}{\abs{I}}$ where $I$ is an $s$-orbit of the vertices of $P_{2\ell+1}^s$ (resp. $P_{2\ell+2}$).
The orbits $\set{e_{w_0} - e_{v_0}, e_{v_0} - e_{w_0}}$ and $\set{e_{w_\ell} - e_{v_\ell}, e_{v_\ell} - e_{w_\ell}}$ correspond to the origin and do not contribute to the description of $P_{2\ell+1}^s$ (resp. $P_{2\ell+2}$).
The remaining orbits yield
\[
P_{2\ell+1}^s = P_{2\ell+2}^s = \conv\left\{\pm\frac{1}{2}(e_{v_i}+e_{w_i}-e_{v_{i+1}}-e_{w_{i+1}}) : 0 \leq i \leq \ell-1\right\}.
\]

One can see that the points $\{e_{v_i}+e_{w_i}-e_{v_{i+1}}-e_{w_{i+1}}\}$ form a lattice basis for the fixed subspace (note that it does not matter whether $v_0$ and $w_0$ are identical or not), and with respect to that basis, $P_{2\ell+1}^s$ is unimodularly equivalent to the cross-polytope of dimension $\ell$ dilated by the factor $\frac{1}{2}.$
By \cite[Theorem~1.4]{BJM2013} and the fact that the Ehrhart series of the interval $[-\frac{1}{2},\frac{1}{2}]$ is $\frac{1+t^2}{(1-t)(1-t^2)}$, the result follows by induction on $\ell$.
\end{proof}

\begin{remark}
For even cycle graphs $C_{2\ell+2}$, there is another type of reflection: one that fixes two antipodal vertices.
For such a reflection $sr \in D_{4\ell+4}$, the fixed sub-polytope $P_{2\ell+2}^{sr}$ cannot be studied using the same method as in Proposition~\ref{prop: odd_cycles}.
The one-element $sr$-orbits are $\set{v_0}$ and $\set{w_\ell}$ and the other orbits are $\set{v_i, w_{i-1}}$.
By a similar argument as above, the vertices of the sub-polytope $P_{2\ell+2}^{sr}$ are of the form
\begin{align*}
&\pm\frac{1}{2}(e_{v_i}+e_{w_{i-1}}-e_{v_{i+1}}-e_{w_i}) \;\;\text{for}\;\;i=1,\ldots,\ell-1, \\
&\pm\frac{1}{2}(e_{v_1}+e_{w_0}-2e_{v_0}) \;\text{and}\; \pm\frac{1}{2}(e_{v_\ell}+e_{w_{\ell-1}}-2e_{w_\ell}).
\end{align*}
For $\ell=1$, this is unimodularly equivalent to a dilated square containing the origin in its interior.
For $\ell>1$, one can cut through the points $\pm(e_{v_1}+e_{w_0}-2e_{v_0})$ and $\pm(e_{v_1}+e_{w_0})$, which yields a subpolytope of $2P_{2\ell+2}^{sr}$ containing the origin and four of its vertices.
Again, this is unimodularly equivalent to a square containing the origin.
Hence, $P_{2\ell+2}^{sr}$ is not unimodularly equivalent to a dilated cross-polytope.
\end{remark}

We have computed the invariant polytopes of the symmetric edge polytope fixed by reflections of $D_{2d}$.
The remaining conjugacy classes are the rotations.
For odd $d$, the irreducible characters of $D_{2d}$ are determined by the following table:
\[
\begin{tabular}{c|ccc}
    & $1$ & $r^k$ & $sr^k$ \\
    \hline
    $\psi_1$ & $1$ & $1$ & $1$ \\
    $\psi_2$ & $1$ & $1$ & $-1$ \\
    $\chi_j$ & $2$ & $2\cos{\frac{2jk\pi}{d}}$ & $0$
\end{tabular}.
\]
where $j$ ranges from $1$ to $\frac{d-1}{2}$ and $k$ ranges from $1$ to $d$.

In general, the fixed polytope $P_{d}^{r^k}$ with respect to a rotation $r^k$ is very difficult to compute directly.
Not only does the description of the vertices of $P_{d}^{r^k}$ depend on the cycle decomposition of the permutation action of $r^k$ on the basis vectors of $\RR^{|V|}$, but also on the adjacency of these vertices in the cycle graph.

However, the rotation $r^k\in D_{2d}$, where $k$ and $d$ are coprime, does not fix any vertex of $C_{d}$. Therefore, the induced action on $P_{d}$ fixes only the origin, whose Ehrhart series is simply a geometric series $\ehr(\{0\}, t) = 1 + t + t^2 + \dots = \frac{1}{1-t}$.
This yields the following result when $d$ is prime.

\begin{theorem}\label{thm: prime cycles}
Let $p\geq 3$ be a prime number.
The $H^\ast$-series $H^\ast_{(p)}$ of $P_{p}$ with respect to the action of the dihedral group $D_{2p}$ is a polynomial of degree $p-1$ and its coefficients $H^\ast_{(p),j}$ are given by
\[
H^\ast_{(p),j} = \frac{1}{2p}
\begin{cases}
(h_j^{(p)}-1+p(g_j^{(p)}+1)) \psi_1 + (h_j^{(p)}-1-p(g_j^{(p)}-1))\psi_2 + (2h_j^{(p)}-2)\chi & \text{if $j$ is even,} \\
(p+h_j^{(p)}-1)\psi_1 + (p+h_j^{(p)}-1)\psi_2 + (2h_j^{(p)}-2) \chi & \text{if $j$ is odd.}
\end{cases}
\]
where $h_j^{(p)}$ follows the notation from Proposition \ref{prop: ohsugi}, $g_j^{(p)}:=\binom{(p-1)/2}{j/2}$, and $\chi=\sum_{j}\chi_j$. In particular, $H^\ast_{(p)}$ is effective.
\end{theorem}

To prove Theorem~\ref{thm: prime cycles}, we require the following technical lemma.

\begin{lemma}\label{lem: technical lemma}
Let $d\geq 3$ be an odd integer and let $0\leq j \leq \frac{d-1}{2}$ be even. Define
\[
 g_j^{(d)} = \binom{(d-1)/2}{j/2}
 \quad \text{and} \quad
 h_j^{(d)} = (-1)^j \sum_{i=0}^j (-2)^i \binom{d}{i}\binom{d-1-i}{j-i}.
\]
Then the inequality
$h_j^{(d)} \geq d\cdot(g_j^{(d)}-1)+1$
holds.
\end{lemma}

\begin{proof}
In the case of $j=0$, the statement follows because $h_{0}^{(d)} = g_{0}^{(d)} = 1$.
Hence, we let $0 < j \leq \frac{d-1}{2}$.
In particular, we have $d \geq 5$.

We start by observing the recurrence relations
\[
 g_j^{(d)} = g_{j-2}^{(d-2)} + g_j^{(d-2)}
 \quad \text{and} \quad
 h_j^{(d)} \geq h_{j-2}^{(d-2)}+2h_{j-1}^{(d-2)}+h_j^{(d-2)}
\]
for $0<j\leq \frac{d-1}{2}$ and $g_0^{(d)}=h_0^{(d)}=1$.
The inequality for $h_j^{(d)}$ is an equality if $j < (d-1)/2$. If $j=\frac{d-1}{2}$ then we get
\[h_{\frac{d-1}{2}}^{(d)} = 4h_{\frac{d-3}{2}}^{(d-2)} > 2h_{\frac{d-3}{2}}^{(d-2)} + 2h_{\frac{d-5}{2}}^{(d-2)}=h_{\frac{d-5}{2}}^{(d-2)} + 2h_{\frac{d-3}{2}}^{(d-2)}+h_{\frac{d-1}{2}}^{(d-2)}. \]

For $j>0$, we prove the statement by induction on odd $d$.
Assume
$
h_j^{(d)} > d\cdot(g_j^{(d)}-1)+1
$.
Then, by the recurrences, we have:
\begin{align*}
    h_j^{(d+2)} &\geq h_{j-2}^{(d)} + 2h_{j-1}^{(d)} + h_j^{(d)} 
    > d(g_{j-2}^{(d)} -1) + 1 + 2 h_{j-1}^{(d)} + d (g_j^{(d)}- 1) + 1  
    = d(g_j^{(d+2)}-2)+2+2h_{j-1}^{(d)}.
\end{align*}
At the same time, we can write:
\[
    (d+2)(g_j^{(d+2)}-1)+1 = d(g_j^{(d+2)}-2)+2 + 2(g_{j-2}^{(d)}+g_j^{(d)})+d-3.
\]
Hence it remains to prove that $h_{j-1}^{(d)}\geq g_{j-2}^{(d)}+g_j^{(d)} + \frac{d-3}{2}$.

Here, by our assumption, we let $j:=2k$ and $d:=2n+1$, where $k \geq 1$, $n\geq 2$ and $2k \leq n$.
Since $h_\ell^{(d)} \geq \binom{d-1}{\ell}$ holds for any $\ell$, we get the following inequalities:
\begin{align*}
    h_{j-1}^{(d)}=h_{2k-1}^{(2n+1)} &\geq \binom{2n}{2k-1}\geq \binom{2n}{k} \geq \binom{n+1}{k}+n-1 \\
    &=\binom{n}{k}+\binom{n}{k-1}+n-1 = g_{j-2}^{(d)} + g_j^{(d)} +\frac{d-3}{2}.
\end{align*}
This concludes the proof.
\end{proof}

\begin{proof}[Proof of Theorem \ref{thm: prime cycles}]
For the reflection $s$, we obtain
\[
\det(I-t \cdot \rho_p(s)) = \det\begin{bmatrix}
1-t & 0 & 0 & \cdots \\
0 & \begin{matrix}
1 & -t \\
-t & 1
\end{matrix} & 0 & \cdots \\
0 & 0 & \begin{matrix}
1 & -t \\
-t & 1
\end{matrix} & \cdots \\
\vdots & \vdots & \vdots & \ddots
\end{bmatrix} = (1-t)(1-t^2)^{\frac{p-1}{2}}.
\]
For the rotation $r$, note that $p$ is odd, so we get
$
\det(I-t \cdot \rho_p (r)) = 1 + (-t)^{p}=1-t^p.
$
Since $p$ is a prime number, recall that the rotation $r$, and any power $r^k$ with $1 \le k \le p-1$, fixes only the origin.
That is $P_p^{r^k} = \{0\}$, and so $\ehr(P_p^r, t) = \frac{1}{1-t}$.
Using this and the description of the Ehrhart series in Propositions~\ref{prop: ohsugi} and \ref{prop: odd_cycles}, we obtain:
\begin{align*}
    H^\ast_{(p)}[t](1) &= h_0^{(p)} + h_1^{(p)} t + \cdots + h_{p-1}^{(p)} t^{p-1},\\
    H^\ast_{(p)}[t](s) &= (1+t^2)^{\frac{p-1}{2}}, \\
    H^\ast_{(p)}[t](r) &= \frac{1-t^p}{1-t} = 1 + t + \cdots + t^{p-1},
\end{align*}
where $h_j^{(p)}$ are the values specified in Proposition \ref{prop: ohsugi}.

Consider now the character of the regular module $\RR D_{2p}$, which is given by $\psi_1+\psi_2+2\sum_{j} \chi_j$.
It is well known that this character evaluates to zero at every element of $D_{2p}$ except at $1$ where it evaluates to $2p$.
Hence, we deduce that the composite character $\chi = \sum_j \chi_j$, obtained by adding together all irreducible two-dimensional characters of $D_{2p}$, is given by:
\[
\begin{tabular}{c|ccc}
    & $1$ & $r^k$ & $sr^k$ \\
    \hline
    $\chi$ & $p-1$ & $-1$ & $0$
\end{tabular}.
\]
The coefficients $H^\ast_{(p),j}$ of the $H^\ast$-series are given by
\[
H^\ast_{(p),j} = \frac{1}{2p}\left\{\begin{matrix}
(h_j^{(p)}-1+p(g_j^{(p)}+1)) \psi_1 + (h_j^{(p)}-1-p(g_j^{(p)}-1))\psi_2 + (2h_j^{(p)}-2)\chi & \text{if $j$ is even,} \\
(p+h_j^{(p)}-1)\psi_1 + (p+h_j^{(p)}-1)\psi_2 + (2h_j^{(p)}-2) \chi & \text{if $j$ is odd.}
\end{matrix}\right.
\]
It remains to show that these quantities are non-negative integers.
Non-negativity follows from Lemma~\ref{lem: technical lemma} and integrality follows immediately from the fact that $H^\ast[t]$ is an element of $R(D_{2p})[[t]]$.
\end{proof}

In the last part of this section, we study the equivariant Ehrhart theory of the order $2$ subgroups associated to the reflections described in Proposition~\ref{prop: odd_cycles}.
Fix the subgroup $S_2 = \{1, s\}$ of $D_{2d}$.
We write $\chi_1$ and $\chi_2$ for the trivial and non-trivial characters of $S_2$ respectively.

\begin{theorem}\label{thm: symm edge poly order 2 group}
Let $d\geq 3$ be an integer and let $\ell = \lfloor d/2 \rfloor$ and $b\in\set{0,1}$ be integers such that $d=2\ell+1+b$.
The equivariant $H^\ast$-series of $P_d$ under the action of $S_2$, denoted $H^\ast_{(d)}[t]$, is a polynomial of degree $d-1$ and its coefficients $H^\ast_{(d),j}$ are given by
\[
H^\ast_{(d),j} = \frac{1}{2}
\left[(h_j^{(d)} + g_j^{(d)})\chi_1 + (h_j^{(d)} - g_j^{(d)})\chi_2\right].
\]
where $h_j^{(d)}$ follows the notation from Proposition \ref{prop: ohsugi} and $g_j^{(d)}$ are the coefficients of the polynomial $(1+t)^{b} (1+t^2)^{\ell} := g_0^{(d)} + g_1^{(d)} t + \cdots + g_{d-1}^{(d)}t^{d-1}$. In particular, $H^\ast_{(d)}[t]$ is effective.
\end{theorem}

\begin{proof}
By a similar argument to the proof of Theorem \ref{thm: prime cycles}, we obtain
\[
\det(I-t\cdot \rho_d(s)) = (1-t)^{1-b} (1-t^2)^ {\ell+b}.
\]
By the description of $\ehr(P_d,t)$ in Proposition~\ref{prop: odd_cycles}, we have:
\begin{align*}
    H^\ast_{(d)}[t](1) &= h_0^{(d)} + h_1^{(d)} t + \cdots + h_{d-1}^{(d)} t^{d-1}, \\
    H^\ast_{(d)}[t](s) &= (1+t)^{b} (1+t^2)^{\ell} = g_0^{(d)} + g_1^{(d)} t + \cdots + g_{d-1}^{(d)}t^{d-1}.
\end{align*}
For the coefficients $H^\ast_{(d),j}$ of the $H^\ast$-series, we obtain
\[
H^\ast_{(d),j} = \frac{1}{2}
\left[(h_j^{(d)} + g_j^{(d)})\chi_1 + (h_j^{(d)} - g_j^{(d)})\chi_2\right].
\]
It remains to show that $H^\ast_{(d)}$ is effective, for which it suffices to show that $h_j^{(d)}\geq g_j^{(d)}$.
If $d$ is odd, this follows directly from Lemma \ref{lem: technical lemma}.
If $d$ is even, we start with the case where $j$ is also even.
We can use that in this case, $g_j^{(d)}=g_j^{(d-1)}$, giving us
\[h_j^{(d)}\geq h_{j-1}^{(d-1)} + h_j^{(d-1)}\geq h_j^{(d-1)}\geq g_j^{(d-1)}= g_j^{(d)}.\]
For the case where $j$ is odd, we may assume without loss of generality that $j\leq \ell-1$.
In this case, we use $g_j^{(d)}=g_{j-1}^{(d)}$ and the fact that $H^\ast_{(d)}(1)$ is unimodal, to conclude
\[h_j^{(d)}\geq h_{j-1}^{(d)} \geq g_{j-1}^{(d)} = g_j^{(d)}.\]
So we have shown that $H^\ast_{(d)}[t]$ is effective, completing the proof.
\end{proof}

\section{Rational cross-polytopes}\label{sec: rational cross polytopes}

Let $k, d \in \ZZ$ be positive integers with $k$ odd and $d \geq 2$. Throughout this section we consider the polytope
\[
P(k,d) = \conv\left\{\pm e_1, \dots, \pm e_{d-1}, \pm \frac k2e_d \right\} \subseteq M_\RR \cong \RR^d.
\]
In this section we prove Theorem~\ref{thm: cross polytope coord reflection} which gives a complete description of the equivariant $H^\ast$-series of $P(k,d)$ under the action of a reflection group. We observe in Example~\ref{ex: cross polytope rational non effective} that a rational analogue of Conjecture~\ref{effectiveness_conjecture} does not hold for rational polytopes with period one.

\smallskip

The Ehrhart series $\ehr(P(k,d), t)$ has the following explicit description.
\begin{proposition}[{An application of \cite[Theorem~1.4]{BJM2013}}] \label{prop: ehrhart series of cross polytope}
For each $k$ odd and $d \geq 2$ we have
\[
\ehr(P(k,d), t) =
(1-t)\ehr([k/2, -k/2], t) \frac{(1+t)^{d-1}}{(1-t)^d} =
\frac{(1 + (k-1)t + kt^2)(1+t)^{d-2}}{(1-t)^{d+1}}.
\]
\end{proposition}

In the following, we will refer to $(1+(k-1)t+kt^2)(1+t)^{d-2}$ by $\Tilde{h}_{P(k,d)}$.
We denote by $G = \{1, \sigma \}$ the group of order two. We fix its two irreducible characters: the trivial character $\chi_1$ and non-trivial character $\chi_2$. Fix some index $i \in [n]$. We let $G$ act on the lattice $\ZZ[e_1, \dots, e_d]$ by a coordinate reflection $\sigma(e_i) = -e_i$ and $\sigma(e_j) = e_j$ for all $j \in [n] \backslash \{i\}$.

\begin{proposition}\label{prop: cross polytope ei reflection}
If $i \in \{1,2, \dots, d-1 \}$, then $H^\ast[t] = \chi_1 \cdot \Tilde{h}_{P(k,d)}(t)$.
\end{proposition}

\begin{proof}
The reflection $\sigma$ acts on $P(k,d)$ by the diagonal matrix $A = \diag(1,\dots, 1, -1, 1, \dots, 1)$ where $-1$ appears in position $i$. Therefore, we may compute $\det(I - tA) = (1-t)^{d-1}(1+t)$.

We proceed by taking cases on $d$; either $d = 2$ or $d > 2$. Fix $d = 2$. In this case, the fixed polytope $P(k,2)^\sigma$ is a line segment $[k/2, -k/2]$ and so its Ehrhart series is
\[
\ehr(P(k,2)^\sigma, t) =
\frac{1+(k-1)t + kt^2}{(1-t)(1-t^2)} =
\frac{\Tilde{h}_{P(k,2)}(t)}{(1-t)\det(I - tA)}.
\]
On the other hand, the identity element $e \in G$ acts by the identity matrix $I$ and so $\det(I - tI) = (1-t)^3$. Clearly, this fixes the entire polytope $P(k,d)$, so its Ehrhart series is given by
\[
\ehr(P(k,2), t) =
\frac{1+(k-1)t + kt^2}{(1-t)^3} =
\frac{\Tilde{h}_{P(k,2)}(t)}{(1-t)\det(I - tI)}.
\]
And so we have that $H^\ast[t] = \chi_1 \cdot \Tilde{h}_{P(k,2)}(t)$ and we are done for the case $d = 2$.

\smallskip

Next, let $d > 2$. The fixed polytope $P(k,d)^\sigma$ is equal to $P(k,d-1)$ in a one-dimension-higher ambient space, and so, by Proposition~\ref{prop: ehrhart series of cross polytope}, its Ehrhart series is given by
\[
\ehr(P(k,d)^\sigma, t) =
\frac{(1+(k-1)t+kt^2)(1+t)^{d-3}}{(1-t)^d} \frac{(1+t)}{(1+t)}=
\frac{\Tilde{h}_{P(k,d)}(t)}{(1-t)\det(I - tA)}.
\]
On the other hand the identity element $e \in G$ fixes the entire polytope $P(k,d)$ and so its Ehrhart series is
\[
\ehr(P(k,d), t) =
\frac{\Tilde{h}_{P(k,d)}(t)}{(1-t)^{d+1}} =
\frac{\Tilde{h}_{P(k,d)}(t)}{(1-t)\det(I - tI)}.
\]
And so it follows that $H^\ast[t] = \chi_1 \cdot \Tilde{h}_{P(k,d)}(t)$ and we are done for the case $d > 2$.
\end{proof}

\begin{proposition}\label{prop: cross poly ed reflection}
If $i=d$, then $H^\ast[t] = \sum_{j = 0}^d (a_j \chi_1 + b_j \chi_2) t^j$ where
\[
a_j = \binom{d-2}{j} + \frac 12 (k+1)\binom{d-1}{j-1}
\text{ and }
b_j = \frac 12(k-1) \binom{d-1}{j-1} - \binom{d-2}{j-1}
\]
and $\binom{n}{k}$ is defined to be zero if $k < 0$ or $k > n$.

\end{proposition}
\begin{proof}
The identity $e \in G$ acts by the identity matrix $I$, hence $\det(I - tI) = (1-t)^d$. So, by Proposition~\ref{prop: ehrhart series of cross polytope}, we have
\[
\ehr(P(k,d), t) =
\frac{(1 + (k-1)t + kt^2)(1+t)^{d-2}}{(1-t)^{d+1}} =
\frac{(1 + (k-1)t + kt^2)(1+t)^{d-2}}{(1-t)\det(I - tI)}.
\]

On the other hand, the reflection acts by the diagonal matrix $A = \diag(1, \dots, 1, -1)$ hence $\det(I - tA) = (1-t)^{d-1}(1+t)$. Observe that the fixed polytope $P(k,d)^\sigma$ is a $(d-1)$-dimensional cross-polytope, therefore we have
\[
\ehr(P(k,d)^\sigma,t) =
\frac{(1+t)^{d-1}}{(1-t)^d} =
\frac{(1+t)^d}{(1-t)\det(I-tA)}.
\]
Write $H^\ast[t]=\sum_{j=0}^d(a_j\chi_1+b_j\chi_2)t^j$ for some $a_j$ and $b_j$. By evaluating $H^\ast[t]$ at each group element $g \in G$, we have $H^\ast[t](g) = \ehr(P(k,d)^g, t)(1-t)\det(I - t\rho(g))$. It follows that
\[
\begin{cases}
a_j+b_j=\binom{d-2}{j}+(k-1)\binom{d-2}{j-1}+ k\binom{d-2}{j-2}, \\
a_j-b_j=\binom{d}{j} = \binom{d-2}{j} + 2 \binom{d-2}{j-1} + \binom{d-2}{j-2}
\end{cases}
\]
for $j \in \{0,1,\ldots,d\}$ where $\binom{n}{k}$ is defined to be zero if $k < n$ or $k > n$. By solving this, we obtain the desired conclusion.
\end{proof}

\begin{example}\label{ex: cross polytope rational non effective}
Consider the case $d = 2$ and $k = 1$. The polytope $P(1,2)$ is given by
\[
P(1,2) = \conv\{(1,0), (-1,0), (0, 1/2), (0,-1/2) \}.
\]
The group $G = \{1, \sigma\}$ acts by a coordinate reflection: $\sigma(e_2) = -e_2$ and $\sigma(e_1) = e_1$. The equivariant Ehrhart $H^\ast$-series is $H^\ast[t] = \chi_1 + (\chi_1 - \chi_2)t + \chi_1 t^2$. In particular, $H^\ast[t]$ is polynomial but not effective since $\chi_1 - \chi_2$ is not the character of a representation of $G$.
\end{example}

\begin{remark}
Consider the dilate of the polytope $2P(1,2)$ with the same group action as in Example~\ref{ex: cross polytope rational non effective}. In this case the equivariant $H^\ast$-series is given by $H^\ast[t] = \chi_1 \cdot (1 + 4t + 3t^2) = \chi_1 \cdot \Tilde{h}_{2P(k,d)}(t)$. The example $P(1,2)$ does not extend to an example of a lattice polytope since all lattice points of $P(1,2)$ are fixed by $G$. However, if $G$ is a non-trivial group acting non-trivially on a full dimensional lattice polytope, then at least one lattice point of $P$ is not fixed by $G$. Concretely, we can say the following about two dimensional polytopes.

Suppose that $G$ is the group of order $2$ and irreducible characters $\chi_1$ and $\chi_2$. Assume $G$ acts on a $2$-dimensional lattice $M$ and let $P$ be a $G$-invariant lattice polytope with a polynomial equivariant $H^\ast$-series given by $H^\ast[t] = \chi_1 + (a \chi_1 + b\chi_2)t + c\chi_1 t^2$ for some $a,b,c \in \ZZ$.
By Corollary~6.7 in \cite{stapledon2011equivariant}, $H^\ast[t]$ is effective.
Moreover, since $\chi_1$ corresponds to a trivial permutation representation and $\chi_1+\chi_2$ corresponds to the regular representation, which is a permutation representation as well, the linear coefficient of $H^\ast[t]$ is itself a permutation representation if $a\geq b\geq 0$.
To see that this is satisfied, one first should notice that $2P^\sigma$ is a lattice polytope by Corollary~5.4 in \cite{stapledon2011equivariant} and so it is either a line segment or a point whose vertices have coordinates lying in $\frac{1}{2}\ZZ$. If $P^\sigma$ is a non-lattice point, then the result follows from a simple computation.
So, by Lemma~7.3 in \cite{stapledon2011equivariant} and our assumption that $H^\ast[t]$ is a polynomial, we only need to consider the case where $P^\sigma$ contain a lattice point. So, it follows that $P^\sigma$ is unimodular equivalent to a line segment $[v, w] \subseteq \RR$ with $v$, $w \in \frac 12 \ZZ$. By taking cases on whether $v$ or $w$ lie in $\ZZ$ we can show that the Ehrhart series has the form
\[\ehr(P^\sigma, t) = \frac{1 + rt + st^2}{(1-t)(1-t^2)}\]
with $r,s\geq 0$.
Evaluating $H^\ast[t]$ at $\sigma$ and comparing coefficients gives us $a-b=r\geq 0$

\end{remark}

Let $G = (\ZZ/2\ZZ)^d = \langle \sigma_1, \sigma_2, \dots, \sigma_d \rangle$ be the group of coordinate reflections of $\RR^d$. Explicitly, for each $i, j \in \{1, 2, \dots, d\}$ we have $\sigma_i(e_i) = -e_i$ and $\sigma_i(e_j) = e_j$ if $i \neq j$. Let $\chi_1$ denote the trivial character of $G$ and $\chi_2$ denote the character satisfying $\chi_2(\sigma_d) = -1$ and $\chi_2(\sigma_i) = 1$ for all $i \in \{1, 2, \dots, d-1\}$. The polytope $P(k,d)$ is invariant under $G$. By Propositions~\ref{prop: cross polytope ei reflection} and \ref{prop: cross poly ed reflection} it follows that the equivariant $H^\ast$-series $H^\ast[t]$ of $P$ is a polynomial whose coefficients are integer multiples of $\chi_1$ and $\chi_2$. Moreover, we obtain the following result.

\begin{theorem}\label{thm: cross polytope coord reflection}
With the setup above, we have $H^\ast[t] = \sum_{j = 0}^d (a_j \chi_1 + b_j \chi_2) t^j$ where
\[
a_j = \binom{d-2}{j} + \frac 12 (k+1)\binom{d-1}{j-1}
\text{ and }
b_j = \frac 12(k-1) \binom{d-1}{j-1} - \binom{d-2}{j-1}
\]
and $\binom{n}{k}$ is defined to be zero if $k < 0$ or $k > n$.
\end{theorem}

\begin{example}
Given two subsets $A$ and $B$ that lie in orthogonal subspaces of $\RR^N$, we denote by $A \oplus B = \conv(A \cup B) \subseteq \RR^N$ the \textit{free sum} of $A$ and $B$. Let $d \geq 3$ and $k = 1$. The polytope $P(k,d)$ has Ehrhart series
\[
\ehr(P(1,d), t) = \frac{(1+t+t^2+t^3)(1+t)^{d-3}}{(1-t)^{d+1}}.
\]
We note that this coincides with the Ehrhart series of the lattice polytope $Q_d \subseteq \RR^3 \times \RR^{d-3}$ given by $Q_d = \conv\{e_1,e_2,e_3, -e_1-e_2-e_3 \} \oplus [-1, 1]^{\oplus (d-3)}$. By a result of Stapledon \cite[Proposition~6.1]{stapledon2011equivariant}, the equivariant $H^\ast$-series of the simplex $S=\conv\{e_1,e_2,e_3, -e_1-e_2-e_3 \}$ is always effective.
If a group $G = \{1, \sigma \}$ acts on $Q_d$ with an action that factors $\sigma(x,y) = (\sigma|_{\RR^3}(x), \sigma|_{\RR^{d-3}}(y)) $ such that $\sigma|_{\RR^{d-3}}$ acts by a coordinate reflection, then the equivariant $H^\ast$-series of $Q_d$ is $(1+t)^{d-3}$ times the $H^\ast$-series of $S$, meaning that it is effective.

On the other hand, if we take the polytope $P(1,d)$ with respect to the action of $G = \{1, \sigma\}$ given by $\sigma(e_d) = -e_d$ and $\sigma(e_i) = e_i$ for all $i \in \{1, \dots, d-1\}$ then the equivariant $H^\ast$-series is not effective.
\end{example}

\bibliographystyle{plain}
\bibliography{references}

\begin{thebibliography}{1}

\bibitem{ardila2019permutahedron}
Federico Ardila, Mariel Supina, and Andr\'{e}s~R. Vindas-Mel\'{e}ndez.
\newblock The equivariant {E}hrhart theory of the permutahedron.
\newblock {\em Proc. Amer. Math. Soc.}, 148(12):5091--5107, 2020.

\bibitem{BJM2013}
Matthias Beck, Pallavi Jayawant, and Tyrrell~B. McAllister.
\newblock Lattice-point generating functions for free sums of convex sets.
\newblock {\em J. Combin. Theory Ser. A}, 120(6):1246--1262, 2013.

\bibitem{beck2007computing}
Matthias Beck and Sinai Robins.
\newblock {\em Computing the continuous discretely}, volume~61.
\newblock Springer, 2007.

\bibitem{curtis1966representation}
Charles~W Curtis and Irving Reiner.
\newblock {\em Representation theory of finite groups and associative
  algebras}, volume 356.
\newblock American Mathematical Soc., 1966.

\bibitem{elia2022techniques}
Sophia Elia, Donghyun Kim, and Mariel Supina.
\newblock Techniques in equivariant ehrhart theory.
\newblock {\em arXiv:2205.05900}, 2022.

\bibitem{isaacs1994character}
I~Martin Isaacs.
\newblock {\em Character theory of finite groups}, volume~69.
\newblock Courier Corporation, 1994.

\bibitem{ohsugi2012smooth}
Hidefumi Ohsugi and Kazuki Shibata.
\newblock Smooth fano polytopes whose ehrhart polynomial has a root with large
  real part.
\newblock {\em Discrete \& Computational Geometry}, 47(3):624--628, 2012.

\bibitem{stanley1980decompositions}
Richard~P. Stanley.
\newblock Decompositions of rational convex polytopes.
\newblock In J.~Srivastava, editor, {\em Combinatorial Mathematics, Optimal
  Designs and Their Applications}, volume~6 of {\em Annals of Discrete
  Mathematics}, pages 333--342. Elsevier, 1980.

\bibitem{stapledon2011equivariant}
Alan Stapledon.
\newblock Equivariant ehrhart theory.
\newblock {\em Advances in Mathematics}, 226(4):3622--3654, 2011.

\end{thebibliography}

\bigskip
\bigskip
\noindent
{\bf Authors' addresses:}

\noindent
Graduate School of Information Science and Technology,
Osaka University, Suita, Osaka 565-0871, Japan\\
E-mail address: {\tt oliver.clarke.crgs@gmail.com}

\noindent
Graduate School of Information Science and Technology,
Osaka University, Suita, Osaka 565-0871, Japan\\
E-mail address:  {\tt higashitani@ist.osaka-u.ac.jp}

\noindent
Graduate School of Information Science and Technology,
Osaka University, Suita, Osaka 565-0871, Japan\\
E-mail address:  {\tt max.koelbl@ist.osaka-u.ac.jp}

\end{document}